\documentclass{amsart}
\usepackage{amsmath,amssymb,amsthm,amsfonts,amscd,mathrsfs}

\usepackage[all]{xy}
\usepackage[dvips]{graphicx}

\newcommand{\imm}{\looparrowright}
\newcommand{\be}{\begin{enumerate}}
\newcommand{\ee}{\end{enumerate}}
\newcommand{\R}{\mathbb{R}}
\newcommand{\Z}{\mathbb{Z}}

\newcommand{\co}{\colon\thinspace}

\newcommand{\stab}{\stackrel{\simeq}{\longrightarrow}}
\newtheorem*{thmI}{Theorem A}
\newtheorem{thm}{Theorem}[section]

\newtheorem{prop}[thm]{Proposition}
\newtheorem{defn}[thm]{Definition}

\newtheorem{lemma}[thm]{Lemma}

\newtheorem{rem}[thm]{Remark}

\begin{document}

\title{On self-intersection invariants}

\author{Mark Grant}

\address{School of Mathematical Sciences, The University of Nottingham,
University Park, Nottingham, NG7 2RD, UK}

\maketitle
\begin{abstract} We prove that the Hatcher-Quinn and Wall invariants of a self-transverse immersion $f\co N^n\imm M^{2n}$ coincide. That is, we construct an isomorphism between their target groups which carries one onto the other. We also employ methods of normal bordism theory to investigate the Hatcher-Quinn invariant of an immersion $f\co N^n\imm M^{2n-1}$.
\end{abstract}

\section{Introduction}

The problem of finding necessary and sufficient conditions for a given (smooth) immersion to be regularly homotopic to an embedding has been considered by many authors, going back to Whitney \cite{Whi}. In favourable cases, complete obstructions can be given in terms of the self-intersection data of the immersion. This is true of Whitney's original trick, which shows that an immersion $f\co N^n\imm M^{2n}$ with $M$ simply-connected is regularly homotopic to an embedding if and only if the algebraic sum of its double points is zero.

The non-simply-connected version of Whitney's trick was used by Wall \cite{Wa66} in the course of his pioneering work on Surgery Theory. To each immersion $f\co N^n\imm M^{2n}$ (where $N$ is now assumed to be simply-connected, but $M$ not necessarily so) Wall describes a complete obstruction to the removal of double points. This invariant (which we denote by $\mu_W(f)$ below) lives in a certain quotient of the integral group ring of $\pi_1(M)$.

Meanwhile, Shapiro \cite{Sha} and Haefliger \cite{Hae} had set about generalising Whitney's trick to higher dimensional self-intersections, using deleted product constructions.  Although their approach essentially reduces the problem to homotopy theory in the so-called meta-stable range, the invariants produced are rather difficult to compute. Later Hatcher and Quinn \cite{HQ} revisited the geometric constructions of Haefliger in the framework of bordism theory. They define, for each immersion $f\co N^n\imm M^m$, a regular homotopy invariant $\mu(f)$ in a certain normal bordism group. When $2m\geq 3(n+1)$ the vanishing of $\mu(f)$ is a necessary and sufficient condition for $f$ to be regularly homotopic to an embedding.

The Hatcher-Quinn invariants have received relatively little attention in the literature (although see the papers of Klein and Williams \cite{KW1}, \cite{KW2}, Munson \cite{Mu}, and Salikhov \cite{Sali}). This can perhaps be attributed to the difficulty of working directly with normal bordism groups, as well as the somewhat sketchy nature of the proofs in \cite{HQ} (although more complete proofs have since been given by Klein and Williams using homotopy-theoretic methods, see \cite[Appendix A]{KW1}).

The current paper has two modest aims. Firstly, we identify the Hatcher-Quinn and Wall invariants of an immersion $f\co N^n\imm M^{2n}$. Secondly, we offer some speculation as to what the analogue of Wall's invariant should be in the case of an immersion $f\co N^n\imm M^{2n-1}$. We remark that in the case of an immersion $f\co S^n\imm\R^{2n-1}$, a very satisfactory answer has been given by Ekholm
\cite{Ek} in terms of Smale invariants. Regular homotopy classes of immersions $f\co M^3\imm \R^5$ have been studied by Saeki-Sz\H ucs-Takase \cite{SST} and by Juh\' asz \cite{J}. In these dimensions the presence of knotted spheres precludes the possibility of describing the regular homotopy class of an immersion in terms of its self-intersection data.

We now give the plan of the paper. After a brief review of normal bordism theory in Section 2, we review the definitions and results of Hatcher-Quinn and Wall in Sections 3 and 4 respectively. In Section 5, we prove the following precise result.

\begin{thmI} Let $f\co N^n\imm M^{2n}$ be a self-transverse immersion, where $N$ is closed and simply-connected and $M$ is connected. Then there is an isomorphism of abelian groups
 \[
 \mathscr{F}\co H_0(\Z_2 ; \Z[\pi]) \stab  \Omega_{0}(P(f,f)_{\Z_2} ; \zeta_{\Z_2})
 \]
 under which $\mathscr{F}\big(\mu_W(f)\big) = \mu(f)$. That is, the Hatcher-Quinn and Wall invariants of $f$ coincide.
\end{thmI}

In Section 6, we use the Gysin sequence in normal bordism to study the Hatcher-Quinn invariant of an immersion $f\co N^n\imm M^{2n-1}$. We aim to construct an analogue of Wall's invariant, residing in a group defined in terms of the first and second homotopy groups of $M$, and depending only on the self-intersection data of $f$. The results in the final section go some way towards realizing this goal.

My sincere thanks go to Andrew Ranicki, without whose considerable encouragement this paper would not have been written, and to the anonymous referee, for suggestions improving exposition.

\section{Normal bordism}

In this section we collect some facts about normal bordism theory. These results may all be found in the paper of Salomonsen \cite{Sal}. Alternative treatments have been given by Dax \cite{Dax} and Koschorke \cite{Ko}. For simplicity, we treat only the absolute bordism groups.\\

Let $X$ be a topological space, and let $\xi = \xi^+ - \xi^-$ be a virtual vector bundle over $X$ (we do not assume that $\xi^+$ and $\xi^-$ are of the same dimension). By an {\em $n$-dimensional $\xi$-manifold over $X$} we mean a triple $\mathscr{M}=(M^n,f,F)$ consisting of a closed $n$-manifold $M^n$, a continuous map $f\co M\to X$, and an equivalence class of vector bundle isomorphisms
\begin{equation}
F\co TM\oplus f^* \xi^-\oplus\varepsilon^r \stab f^*\xi^+\oplus\varepsilon^s
\end{equation}
where $r$ and $s$ are suitable integers (here and elsewhere $\varepsilon$ denotes a trivial bundle of the stated dimension). The equivalence relation is generated by stabilisation and homotopy of bundle isomorphisms.

The negative of $\mathscr{M}$ is the triple $-\mathscr{M} = (M,f,-F)$ where
\begin{equation}
-F = F\oplus(-1)\co TM\oplus f^* \xi^-\oplus\varepsilon^r\oplus\varepsilon^1 \stab f^*\xi^+\oplus\varepsilon^s\oplus\varepsilon^1.
\end{equation}
 If $\mathscr{M} = (M,f,F)$ and $\mathscr{N}= (N,g,G)$ are $\xi$-manifolds over $X$, their disjoint union defines a $\xi$-manifold $\mathscr{M} + \mathscr{N} = (M\sqcup N, f\sqcup g, F\sqcup G)$. The empty $\xi$-manifold will be denoted $\mathscr{O} = (\emptyset,\emptyset,\emptyset)$.

We introduce a bordism relation on the set of $n$-dimensional $\xi$-manifolds $\mathscr{M} = (M,f,F)$ over $X$, as follows. We say that $\mathscr{M}\sim\mathscr{O}$ if there exists triple $(W,\varphi,\Phi)$ consisting of a compact $(n+1)$-manifold $W$ with boundary $\partial W =M$, a continuous map $\varphi\co W\to X$ such that $\varphi|_{\partial W} = f$, and a bundle isomorphism
 \begin{equation}
 \Phi\co TW\oplus \varphi^*\xi^-\oplus\varepsilon^r \stab \varphi^*\xi^+\oplus\varepsilon^s
 \end{equation}
whose restriction to $\partial W$ is equivalent to $F$ (here we use the inward pointing normal vector to make the identification $TW|_{\partial W}\cong TM\oplus\varepsilon^1$). Two $\xi$-manifolds $\mathscr{M}$ and $\mathscr{N}$ are {\em bordant}, written $\mathscr{M}\sim\mathscr{N}$, if $\mathscr{M} - \mathscr{N} \sim \mathscr{O}$. Bordism is an equivalence relation, and the set of bordism classes of $n$-dimensional $\xi$-manifolds over $X$ is denoted $\Omega_n(X;\xi)$. A group structure on $\Omega_n(X;\xi)$ is defined by setting
\begin{equation}
[\mathscr{M}]+[\mathscr{N}]=[\mathscr{M}+\mathscr{N}],\qquad -[\mathscr{M}]=[-\mathscr{M}],\qquad 0 =[\mathscr{O}].
\end{equation}
The resulting abelian group is called the {\em $n$-th normal bordism group of $X$ with coefficients in $\xi$}.

The normal bordism groups are functorial with respect to morphisms of virtual bundles. Let $\overline{h}\co \zeta\to \xi$ be a morphism of virtual bundles covering $h\co Y\to X$, and let $\mathscr{M}=(M, f, F)$ be a $\zeta$-manifold over $Y$. Then the triple $h_*\mathscr{M} = (M,h\circ f, F)$ defines a $\xi$-manifold over $X$ (here we use the canonical isomorphism $\zeta \cong h^*\xi$). This induces a homomorphism of abelian groups
\begin{equation}
h_*\co \Omega_n(Y;\zeta) \to \Omega_n(X; \xi),\qquad h_*[\mathscr{M}] = [h_*\mathscr{M}].
\end{equation}

The normal bordism groups enjoy many properties analogous to the Eilenberg-Steenrod axioms for singular homology. Here we recall a subset of these which will be needed in the sequel.\\

\noindent{\em Dimension zero.}~~There is an isomorphism
\begin{equation}
\Omega_0(X;\xi) \stab H_0(X;\Z(\xi)),
\end{equation}
where $\Z(\xi)$ denotes the local system of integer coefficients twisted by $w_1(\xi)$. In particular $\Omega_0(X;\xi)$ is a direct sum over the path components of $X$ of groups isomorphic to $\Z$ or $\Z_2$, depending on whether the restriction of $\xi$ to the corresponding component is orientable or not.\\

\noindent{\em Homotopy invariance I.}~~ Let $H$ be a homotopy between maps $h_0 , h_1\co Y\to X$. Then the following diagram commutes,
\begin{equation}
\xymatrix{
\Omega_n(Y ; h_0^*\xi) \ar[rd]^{(h_0)_*} \ar@{<->}[dd]^{\Theta_H} & \\
 & \Omega_n(X ; \xi) \\
\Omega_n(Y ; h_1^*\xi) \ar[ur]_{(h_1)_*}
} \end{equation}
where the correspondence $\Theta_H$ is given by the isomorphism $h_0^*\xi\cong h_1^*\xi$ determined by $H$.\\

\noindent{\em Homotopy invariance II.}~~Let $h\co Y\to X$ be a map such that
\begin{equation}
h_*\co \pi_i(Y,y_0)\to \pi_i(X,h(y_0))
\end{equation}
is an isomorphism for $i\leq n$ and an epimorphism for $i= n+1$, with respect to any choice of base point $y_0\in Y$. Then the induced map
\begin{equation}
h_*\co\Omega_i(Y; h^*\xi)\to \Omega_i(X;\xi)
\end{equation}
is an isomorphism for $i\leq n$ and an epimorphism for $i= n+1$.\\

\noindent{\em Gysin sequence.}~~Let $\nu$ be an orthogonal vector bundle over $X$ of rank $k$, with associated sphere bundle $p\co S\nu\to X$. There is a long exact sequence
\begin{equation}
\cdots\to\Omega_n (S\nu ; p^*\xi)\stackrel{p_*}{\longrightarrow} \Omega_n(X;\xi) \stackrel{e(\nu)}{\longrightarrow} \Omega_{n-k}(X;\xi - \nu) \stackrel{w(\nu)}{\longrightarrow} \Omega_{n-1}(S\nu;\xi)\to\cdots
\end{equation}
of normal bordism groups. The homomorphism $e(\nu)$ is called the {\em Euler mapping}.\\

We shall need to understand the Euler mapping in more detail. Let $[M,f,F]\in \Omega_n(X;\xi)$. Let $s\co M\to Ef^*\nu$ be a section transverse to the zero section $M\subseteq Ef^*\nu$. Then the zeroes of $s$ form an $(n-k)$-dimensional submanifold $N\subseteq M$. Let $g = f|_N\co N\to X$; then the normal bundle of $N$ in $M$ is isomorphic to $g^*\nu$, and restriction of $F$ to $N$ gives a bundle isomorphism
\begin{equation}
G\co TN\oplus g^*\nu\oplus g^*\xi^-\oplus\varepsilon^r \stab g^*\xi^+\oplus\varepsilon^s.
\end{equation}
We then have $e(\nu)([M,f,F]) = [N,g,G]\in \Omega_{n-k}(X;\xi - \nu)$.\\

A particular case of interest to us is the Gysin sequence associated to a double cover. Let $\pi\co \widetilde{X}\to X$ be a double cover, and let $\lambda$ be the associated line bundle over $X$, which has $S\lambda =  \widetilde{X}$. Writing $\widetilde\xi$ for $\pi^*\xi$, we obtain a Gysin sequence
\begin{equation}
\cdots\to \Omega_{n+1}(\widetilde{X};\widetilde{\xi}) \stackrel{\pi_*}{\longrightarrow}\Omega_{n+1}(X;\xi) \stackrel{e(\lambda)}{\longrightarrow} \Omega_n(X;\xi - \lambda) \stackrel{w(\lambda)}{\longrightarrow}\Omega_n (\widetilde{X} ; \widetilde\xi)\to \cdots
\end{equation}
where the map $w(\lambda)$ is induced by taking double covers. We refer the reader to \cite[Chapitre I.8]{Dax} and \cite[Section 10]{Sal} for more details.

We shall apply the Gysin sequence of a double cover to obtain information about the low dimensional $\Z_2$-equivariant normal bordism groups of a space with involution. First we recall some terminology and notation. Let $G$ be a group, and let $Y$ be a space on which $G$ acts. A {\em $G$-vector bundle} over $Y$ is a $G$-equivariant map $\xi\to Y$ which is a vector bundle in the usual sense, and such that each $\tau\in G$ induces a vector bundle map $\overline\tau\co \xi \to \xi$. A {\em virtual $G$-bundle} over $Y$ is a formal difference $\xi = \xi^+-\xi^-$ of $G$-vector bundles over $Y$.

The {\em Borel space} of $Y$ is the quotient $Y_G:= EG\times_G Y$ of $EG\times Y$ by the diagonal $G$-action, where $EG$ is a contractible space on which $G$ acts freely. This construction is functorial; in particular a $G$-vector bundle $\xi$ over $Y$ gives rise to a vector bundle $\xi_G$ over $Y_G$ of the same dimension. Similarly, given a virtual $G$-bundle $\xi= \xi^+ -\xi^-$ over $Y$ we get a virtual bundle
$\xi_G = \xi^+_G - \xi^-_G$ over $Y_G$. We then define the {\em $n$-th $G$-equivariant normal bordism group of $Y$ with coefficients in $\xi$} to be the group $\Omega_n(Y_G ; \xi_G)$.

When $G=\Z_2$ the quotient map $\pi\co E\Z_2\times Y \to Y_{\Z_2}$ is a double cover, and is homotopically equivalent to the map $i\co Y\to Y_{\Z_2}$ given by $i(y)=[e,y]$ for some choice of base-point $e\in E\Z_2$.

\begin{prop}\label{zero} Let $\xi$ be a virtual $\Z_2$-bundle over a space with involution $t\co Y\to Y$. Let $\Omega_0(Y;\xi)$ have the $\Z_2$-module structure given by
\begin{equation}
t_*\co \Omega_0(Y; \xi)\to \Omega_0(Y;\xi).
\end{equation}
Then the map $i\co Y\to Y_{\Z_2}$ induces an isomorphism of abelian groups
\begin{equation}
i_*\co H_0\big(\Z_2 ; \Omega_0(Y;\xi)\big) = \frac{\Omega_0(Y;\xi)}{\{a - t_* a\mid a\in\Omega_0(Y;\xi)\}} \stab \Omega_0(Y_{\Z_2} ; \xi_{\Z_2}).
\end{equation}
\end{prop}
\begin{proof}
The Gysin sequence for the double cover $Y\simeq E\Z_2\times Y\to Y_{\Z_2}$ ends
\begin{equation}
\cdots \longrightarrow \Omega_0(Y_{\Z_2};\xi_{\Z_2} - \lambda) \stackrel{w(\lambda)}{\longrightarrow} \Omega_0(Y;\xi) \stackrel{i_*}{\longrightarrow} \Omega_0(Y_{\Z_2} ; \xi_{\Z_2}) \longrightarrow 0.
\end{equation}
It is not difficult to check that the image of $w(\lambda)$ is the subgroup $$\{a -  t_* a\mid a\in\Omega_0(Y;\xi)\}\subseteq\Omega_0(Y;\xi).$$
\end{proof}

\section{Hatcher-Quinn invariants}

In this section we recall some definitions and results of Hatcher and Quinn \cite{HQ}, who defined a regular homotopy invariant $\mu(f)$ which vanishes if (and in a certain dimension range, only if) the immersion $f$ is regularly homotopic to an embedding. We use the conventions for normal bordism groups set out in the previous section.\\

Let $f\co N^n\imm M^m$ be an immersion. The homotopy pullback
\begin{equation}
P(f,f) =\{ (x,\gamma ,y)\in N\times M^I \times N \mid f(x)=\gamma(0)\mbox{ and } f(y) = \gamma(1)\}
\end{equation}
fits into a homotopy commutative diagram
\begin{equation}
\xymatrix{
  P(f,f) \ar[r]^-{p_2} \ar[d]^{p_1} \ar[rd]^{p} &  N \ar[d]^{f} \\
  N \ar[r]^-{f}  &  M
 }
 \end{equation}
 where $p_1(x,\gamma,y)=x$, $p_2(x,\gamma,y)=y$ and $p(x,\gamma,y)=\gamma(1/2)$. It has the following universal property: if $T$ is another space with maps $\rho_1,\rho_2\co T\to N$ such that $f\rho_1 \simeq f\rho_2$, then there is a map $\phi\co T\to P(f,f)$, unique up to homotopy, such that $p_1\phi\simeq\rho_1$ and $p_2\phi\simeq\rho_2$.

Now suppose that $f\co N^n\imm M^m$ is self-transverse, and $N$ is closed. Then the space
 \begin{equation}
 \overline{\Sigma}(f) = \{ (x,y)\in N\times N \mid f(x)=f(y)\mbox{ and } x\neq y \}.
 \end{equation}
 is a closed submanifold of $N\times N$ of dimension $2n-m$, the so-called {\em self-intersection manifold} of $f$. The projections $\rho_1,\rho_2\co \overline{\Sigma}(f)\to N$ ensure that there is a homotopy commutative diagram
\begin{equation}
\xymatrix{
 \overline{\Sigma}(f)\ar[rd]^{\bar{\phi}} \ar@/^/[rrd]^{\rho_2} \ar@/_/[rdd]_{\rho_1} &    &     \\
         &   P(f,f) \ar[r]^{p_2} \ar[d]^{p_1} \ar[rd]^{p} &  N \ar[d]^{f} \\
   &   N \ar[r]^-{f}  &  M,
 }
 \end{equation}
 where $\bar{\phi}(x,y) = (x,c_{f(x)},y)$ for $c_{f(x)}$ the constant path at $f(x)$.

 Let $\psi\co \overline{\Sigma}(f)\to M$ be the composition $p\bar{\phi}$, so $\psi(x,y) = f(x)=f(y)$. The self-intersection manifold fits into a pullback diagram
\begin{equation}\label{pullback}
\xymatrix{
\overline{\Sigma}(f) \ar@{^{(}->}[d]^{i} \ar[rr]^\psi & & M \ar[d]^{\triangle_M} \\
N\times N - \triangle_N (N) \ar[rr]^{f\times f|} & & M\times M,
}
\end{equation}
where for a space $X$ we denote by $\triangle_X\co X\to X\times X$ the diagonal map $x\mapsto (x,x)$. The embedding $i\co \overline{\Sigma}(f)\hookrightarrow N\times N$ factors as $(\rho_1\times\rho_2)\triangle_{\overline{\Sigma}(f)}$. We therefore have a sequence of vector bundle isomorphisms
\begin{align*}
T\overline{\Sigma}(f)\oplus \psi^*TM      & \cong  T\overline{\Sigma}(f)\oplus \psi^*\nu_{\triangle_M} \\
                                          & \cong  T\overline{\Sigma}(f)\oplus \nu_i \\
                                          & \cong  i^*T(N\times N - \triangle N) \\
                                          & \cong  \rho_1^*TN\oplus\rho_2^*TN,
\end{align*}
where $\nu$ denotes a normal bundle. Now note that each of the maps $\psi$, $\rho_1$ and $\rho_2$ factor through $\bar{\phi}\co \overline{\Sigma}(f)\to P(f,f)$, and so we have constructed a bundle isomorphism
 \begin{equation}
 \overline{\Phi}\co T\overline{\Sigma}(f)\oplus \bar{\phi}^*p^*TM\stab  \bar{\phi}^*(p_1^*TN\oplus p_2^*TN).
 \end{equation}
It follows that the self-intersection manifold of $f$ represents an element
\begin{equation}
[ \overline{\Sigma}(f), \bar{\phi}, \overline{\Phi}]\in \Omega_{2n-m}( P(f,f) ; \zeta),\qquad \zeta = p_1^*TN\oplus p_2^*TN - p^*TM.
\end{equation}
 In order that we do not count each double point twice, however, we must factor out by the action of the cyclic group $\Z_2$, which acts on all the manifolds in diagram (\ref{pullback}) by swapping factors. In particular $\Z_2$ acts freely on $\overline{\Sigma}(f)$, with quotient
 \begin{equation}
\Sigma(f) = \overline{\Sigma}(f)/\Z_2 = \{ [x,y]\mid (x,y)\in\overline{\Sigma}(f)\}
 \end{equation}
the so-called {\em double-point manifold} of $f$. Let $e\co \overline{\Sigma}(f)\to E\Z_2 = S^\infty$ classify the double cover $\pi\co\overline{\Sigma}(f)\to \Sigma(f)$, and define a map
\begin{equation}
\phi\co \Sigma(f) \to P(f,f)_{\Z_2},\qquad \phi[x,y] = [e(x,y) , \bar\phi (x,y)]=[e(x,y) ,(x, c_{f(x)}, y)].
\end{equation}
There is an involution
\begin{equation}
t\co P(f,f)\to P(f,f),\qquad t(x,\gamma, y) = (y, \overline\gamma, x),\quad \overline\gamma(t) = \gamma(1-t)
\end{equation}
which is covered by the bundle involutions
$$ \begin{array}{cc}
 \bar{t}\co p^*TM\to p^*TM,  &  \bar{t}(v)= -v,\\
 \bar{t}\co p_1^*TN\oplus p_2^*TN\to p_1^*TN\oplus p_2^*TN,  &  \bar{t}(v_1,v_2)= (v_2,v_1).
 \end{array}$$
Factoring out by the $\Z_2$-action, we find that $\overline{\Phi}$ induces a stable bundle isomorphism
\begin{equation}
\Phi\co T\Sigma(f)\oplus \phi^*(p^*TM)_{\Z_2} \stab \phi^* (p_1^*TN\oplus p_1^*TN)_{\Z_2}.
\end{equation}
  \begin{defn}[Hatcher-Quinn \cite{HQ}]\label{HQdef}  Let $f\co N^n\imm M^m$ be a self-transverse immersion with $N$ closed. The {\em Hatcher-Quinn invariant} of $f$ is the normal bordism class
\begin{equation}
\mu(f) =[\Sigma(f), \phi, \Phi] \in \Omega_{2n-m}\big( P(f,f)_{\Z_2} ; \zeta_{\Z_2}\big),
\end{equation}
where $\zeta_{\Z_2}$ is the virtual vector bundle $(p_1^*TN\oplus p_2^*TN)_{\Z_2} - (p^*TM)_{\Z_2}$.
\end{defn}
\begin{rem} {\em If $f$ is not self-transverse, then we define $\mu(f)=\mu(f')$, where $f'\co N\imm M$ is a self-transverse immersion regularly homotopic to $f$. This is well-defined by the following result.}
\end{rem}
\begin{thm}[{Hatcher-Quinn \cite[Theorem 2.3]{HQ}}]\label{HQ}
The class $\mu(f)$ is a regular homotopy invariant. If $2m\geq 3(n+1)$ and $\mu(f)=[\mathscr{N}]$ for some singular $\zeta_{\Z_2}$-manifold $\mathscr{N}=(N^{2n-m}, \gamma, \Gamma)$ in $P(f,f)_{\Z_2}$, then $f$ is regularly homotopic to an immersion $g$ with $\Sigma(g) = N$. In particular, $\mu(f)=0$ if and only if $f$ is regularly homotopic to an embedding.
\end{thm}
For an alternative approach to this result, see the papers of Klein and Williams \cite{KW1,KW2} on homotopical intersection theory.

\section{Wall's invariant}

In order to investigate the possibility of performing surgery in the middle dimension on non-simply-connected manifolds, C.\ T.\ C.\ Wall \cite{Wa66} defined an obstruction to a given immersion $f\co S^n \imm M^{2n}$ being regularly homotopic to an embedding. Wall's obstruction is complete when $n\geq 3$. In this section we briefly recall the construction, following \cite{Wa66} (see also \cite[Chapter 5]{Wa99}).\\

Wall's invariant $\mu_{W}(f)$ for a self-transverse immersion $f\co N^n\to M^{2n}$, where $N$ is closed and simply-connected and $M$ is connected, may be described as follows. Choose once and for all a base-point $n_0\in N$ for which $f^{-1}(\{f(n_0)\}) = \{ n_0\}$, and let $m_0=f(n_0)$ be the base-point of $M$. Wall's obstruction lives in a quotient of the integral group ring $\Z[\pi]$ of $\pi=\pi_1(M,m_0)$. In particular, if $w\co\pi\to\{\pm 1\}$ is the orientation character of $M$, then we may define an involution on the group ring,
\begin{equation}\label{involution}
\overline{(\;)}\co\Z[\pi]\to\Z[\pi],\qquad \sum_{\sigma\in\pi} n_\sigma \sigma \mapsto  \sum_{\sigma\in\pi} (-1)^nw(\sigma)n_\sigma \sigma^{-1}.
\end{equation}
This makes $\Z[\pi]$ a $\Z_2$-module, and $\mu_{W}(f)$ will be an element of the group of co-invariants
\begin{equation}
 H_0(\Z_2 ; \Z[\pi]) = \frac{\Z[\pi]}{\{ a - \overline{a}\mid a\in\Z[\pi]\} }.
\end{equation}

The above conditions on $f\co N\imm M$ ensure that the self-intersection $\overline{\Sigma}(f)$ and the double-point manifold $\Sigma(f)$ each consist of a finite number of points. Each double point $[x,y]\in \Sigma(f)$ may be lifted to a self-intersection $(x,y)\in \overline{\Sigma}(f)$ by an arbitrary choice of ordering. For each self-intersection we define an element $\sigma_{(x,y)}\in\pi$ and a sign $\varepsilon_{(x,y)}\in\{\pm 1\}$ as follows. Choose paths $\gamma_x,\gamma_y$ in $N$ from $n_0$ to $x$ and $y$ respectively which avoid other self-intersection points of $f$. Then $\sigma_{(x,y)}\in \pi$ is defined to be the homotopy class of the loop $f\gamma_x\cdot f\overline{\gamma_y}$ in $M$ based at $m_0$. Note that changing the order of $x$ and $y$ reverses the loop, so $\sigma_{(y,x)}={\sigma_{(x,y)}}^{-1}$. To define the sign, fix orientations of $N$ at $n_0$ and $M$ at $m_0$. The tangent spaces $TN_x$ and $TN_y$ become oriented by transport along $\gamma_x$ and $\gamma_y$ of the orientation of $TN_{n_0}$. Set $\varepsilon_{(x,y)}$ to equal $1$ if the orientation of $df(TN_x)$ followed by that of $df(TN_y)$ agrees with the transport of the orientation of $TM_{m_0}$ along $f\gamma_x$, and equal to $-1$ otherwise. Note that $\varepsilon_{(y,x)}=(-1)^n w(\sigma_{(x,y)})\varepsilon_{(x,y)}$.

\begin{defn}  Let $f\co N^n\imm M^{2n}$ be a self-transverse immersion, with $N$ closed and simply-connected and $M$ connected. The {\em Wall invariant} of $f$ is the well-defined class $\mu_{W}(f)\in H_0(\Z_2 ; \Z[\pi])$ represented by the finite sum
\begin{equation}
\tilde\mu_{W}(f) = \sum_{[x,y]\in \Sigma(f)} \varepsilon_{(x,y)} \sigma_{(x,y)} \in \Z[\pi] .
\end{equation}
\end{defn}
\begin{thm}[{Wall \cite[Theorem 3.1]{Wa66}}]\label{Wall}
The class $\mu_W(f)$ is a regular homotopy invariant. If $f$ is regularly homotopic to an embedding, then $\mu_{W}(f)=0$. Conversely, if $n\geq 3$ and $\mu_{W}(f)=0$, then $f$ is regularly homotopic to an embedding.
\end{thm}

\section{Proof of Theorem A}

In this section we prove that the Hatcher-Quinn and Wall invariants of a self-transverse immersion $f\co N^n\imm M^{2n}$, where $N$ is closed and simply-connected and $M$ is connected, reside in isomorphic groups and correspond under this isomorphism.\\

Let $\Z[\pi]$ be the the $\Z_2$-module described in Section 4, where $\pi = \pi_1(M,m_0)$. Recall that the normal bordism group $\Omega_0(P(f,f) ;\zeta)$ (see Section 3) also has the structure of a $\Z_2$-module, given by the involution
\begin{equation}
t_*\co \Omega_0(P(f,f) ;\zeta) \to \Omega_0(P(f,f) ;\zeta),
\end{equation}
where $t\co P(f,f)\to P(f,f)$ is the involution $t(x,\gamma,y) = (y,\overline{\gamma},x)$.
\begin{lemma}\label{modules}
There is an isomorphism of $\Z_2$-modules
\begin{equation}
\chi\co\Z[\pi] \stab \Omega_0(P(f,f) ;\zeta).
\end{equation}
\end{lemma}
\begin{proof}
Consider the fibration $(p_1,p_2)\co P(f,f)\to N\times N$ with fibre $\Lambda M = \Lambda(M,m_0)$ the based loop space of $M$. Since $N$ is simply-connected, the fibre inclusion
\begin{equation}
\iota\co \Lambda M\to P(f,f),\qquad \iota(\gamma) = (n_0, \gamma, n_0)
\end{equation}
induces an isomorphism $\iota_*\co \pi_0(\Lambda M) \stab \pi_0(P(f,f))$, and hence induces an isomorphism
\begin{equation}
\iota_*\co \Omega_0(\Lambda M ; \iota^*\zeta) \stab\Omega_0(P(f,f) ; \zeta).
\end{equation}
Now $\iota^*\zeta = c^*TN\oplus c^*TN - \mathrm{ev}^*TM$, where $c\co \Lambda M\to N$ is constant at $n_0$ and the evaluation map $\mathrm{ev}\co \Lambda M\to M$ given by $\mathrm{ev}(\gamma)= \gamma(1/2)$ is null-homotopic via the homotopy $\gamma \mapsto \gamma\big( (1-t)1/2\big)$. Hence $\iota^*\zeta$ is a trivial virtual bundle, and in particular is orientable over each path-component $\Lambda M_\sigma \subseteq \Lambda M$. Thus there are isomorphisms of abelian groups
\begin{equation}
\Z[\pi]\cong \Omega_0(\Lambda M ; \iota^*\zeta) \cong\Omega_0(P(f,f) ; \zeta).
\end{equation}
We give an explicit isomorphism $\chi\co \Z[\pi] \stab \Omega_0(P(f,f); \zeta)$ by choosing a generator $\chi(\sigma)\in\Omega_0(P(f,f) ; \zeta)$ for each $\sigma\in\pi$, and show that $\chi$ is a map of $\Z_2$-modules.

Let $\gamma$ be a loop in $M$ representing $\sigma$. Fix orientations for the tangent spaces $TN_{n_0}$ and $TM_{m_0}$. These induce orientations of $TM_{\gamma(1/2)}$ by parallel transport along the first half of $\gamma$, and of $TN_{n_0}\oplus TN_{n_0}$ by direct sum. We then set
\begin{equation}
\chi(\sigma) = [P^0,(n_0,\gamma,n_0), \Xi], \qquad \Xi\co TM_{\gamma(1/2)} \stab TN_{n_0}\oplus TN_{n_0},
\end{equation}
where $P^0$ is a point and $\Xi$ is orientation preserving. It is easy to see that $\chi(\sigma)\in \Omega_0(P(f,f);\zeta)$ does not depend on the choices of $\gamma$ and $\Xi$.

Let $t\co P(f,f)\to P(f,f)$ be the involution. In order to show that $\chi$ is a $\Z_2$-module map, we must show that $t_*\chi(\sigma) = (-1)^nw(\sigma)\chi(\sigma^{-1})$. Now $t_*\chi(\sigma)=[P,(n_0,\overline\gamma, n_0), \Psi]$, where $\Psi$ is the vector space isomorphism determined by the diagram
\begin{equation}
\xymatrix{
TM_{\overline\gamma(1/2)} \ar[d]^{\bar{t}} \ar[r]^-{\Psi} &  TN_{n_0}\oplus TN_{n_0} \ar[d]^{\bar{t}} \\
TM_{\gamma(1/2)} \ar[r]^-{\Xi} &  TN_{n_0}\oplus TN_{n_0}.
}
\end{equation}
Orient $TM_{\overline\gamma(1/2)}$ by parallel transport along the first half of $\overline{\gamma}$. We wish to determine the sign of the linear map $\Psi$. If $TM_{\gamma(1/2)}$ is oriented by transport along the {\em first} half of $\gamma$, then the linear map $\bar{t}=(-1)\co TM_{\overline\gamma(1/2)}\to TM_{\gamma(1/2)}$ has sign $(-1)^{2n}w(\sigma) = w(\sigma)$. The map $\Xi$ has sign $+1$. The map $\bar{t}\co TN_{n_0}\oplus TN_{n_0}\to TN_{n_0}\oplus TN_{n_0}$ which swaps factors has sign $(-1)^{n^2}=(-1)^n$. Thus $\Psi$ has sign $(-1)^nw(\sigma)$, and $t_*\chi(\sigma) = (-1)^nw(\sigma)\chi(\sigma^{-1})$ as claimed.
\end{proof}

Combining this lemma with Proposition \ref{zero}, we have group isomorphisms
\begin{equation}
\xymatrix{
H_0(\Z_2 ; \Z[\pi])\ar[r]^-{\chi_*} & H_0\big(\Z_2 ; \Omega_0(P(f,f) ; \zeta)\big)\ar[r]^-{i_*} &  \Omega_0\big( P(f,f)_{\Z_2} ; \zeta_{\Z_2}),
}\end{equation}
where $i\co P(f,f)\to P(f,f)_{\Z_2}$ is given by $i(x,\gamma, y) = [e, (x,\gamma, y)]$ for some base point $e\in E\Z_2$. Set $\mathscr{F} = i_*\circ\chi_*$. The proof of Theorem A is completed by the following lemma.
\begin{lemma}
$\mathscr{F}\big(\mu_W(f)\big) = i_*\left[\chi\,\tilde{\mu}_W(f)\right] = \mu(f).$
\end{lemma}
\begin{proof}
For each double point $[x,y]\in\Sigma(f)$ we choose a lift $(x,y)\in \overline{\Sigma}(f)$, and paths $\gamma_x$ and $\gamma_y$ in $N$ from $n_0$ to $x$ and $y$ respectively. Then
\begin{equation}
\chi\,\tilde{\mu}_W(f) = [\Sigma(f) , \psi, \Upsilon]\in \Omega_0(P(f,f); \zeta)=\Omega_0(P(f,f); i^*\zeta_{\Z_2}),
\end{equation}
where $\psi[x,y] = (n_0 , f\gamma_x\cdot f\overline{\gamma_y}, n_0)$ and over $[x,y]\in \Sigma(f)$ the stable isomorphism
\begin{equation}
\Upsilon\co TM_{f(x)} \stab TN_{n_0}\oplus TN_{n_0}
\end{equation}
has sign $\varepsilon_{(x,y)}$ (see Section 4). So
\begin{equation}
i_*\,\chi_*\,\mu_W(f) = [\Sigma(f) , i\circ\psi, \Upsilon]\in \Omega_0(P(f,f)_{\Z_2}; \zeta_{\Z_2}).
\end{equation}
We next observe that the maps $i\circ \psi, \phi\co \Sigma(f) \to P(f,f)_{\Z_2}$, given by
\begin{equation}
i\circ\psi [x,y] = [e, (n_0 , f\gamma_x\cdot f\overline\gamma_y , n_0)],\qquad \phi[x,y] = [e(x,y), (x,c_{f(x)},y)]
\end{equation}
are homotopic. For each $[x,y]\in\Sigma(f)$ choose a path $\omega_{(x,y)}\co I\to E\Z_2$ from $e$ to $e(x,y)$. For any path $\gamma\co I\to N$ and $t\in I$ define a re-parameterized path $\gamma^t$ (whose image is $\gamma([t,1])$) by setting $\gamma^t(s) = \gamma\big( (1-t)s + t\big)$.  Now the desired homotopy $H\co \Sigma(f) \times I \to P(f,f)_{\Z_2}$ is defined by
\begin{equation}
H([x,y],t) = \bigg[ \omega_{(x,y)}(t), \bigg(\gamma_x(t) ,f\gamma_x^t \cdot f\overline{\gamma_y^t} , \gamma_y(t)\bigg)\bigg].
\end{equation}

By the first property of homotopy invariance of the bordism groups in Section 2, to complete the proof it suffices to check that for each double point $[x,y]\in\Sigma(f)$ the diagram of vector space isomorphisms
\begin{equation}
\xymatrix{
TM_{f(x)} \ar[d]_{\mathrm{id}} \ar[r]^-{\Upsilon} & TN_{n_0}\oplus TN_{n_0} \ar[d]_{\gamma_*} \\
TM_{f(x)} \ar[r]^-{\Phi} & TN_x\oplus TN_y
}
\end{equation}
commutes up to sign. Here the vertical isomorphisms are those induced by the homotopy $H$, and $\gamma_*$ stands for parallel transport along $\gamma_x$ on the first summand and $\gamma_y$ on the second. The isomorphism $\Phi$ is described in Section 3, and may be seen to have inverse given by $(v_1,v_2)\mapsto df_x(v_1) - df_y(v_2)$. The diagram commutes up to sign by the definition of $\varepsilon_{(x,y)}$. This completes the proof of the lemma, and of Theorem A.
\end{proof}

\section{The case $f\co N^n \imm M^{2n-1}$}

In this final section, we offer some speculative remarks concerning the case of an immersion $f\co N^n\imm M^{2n-1}$ where $N$ is closed and simply-connected and $M$ is connected. \\

In this case, if $n\geq 5$ then $f$ is regularly homotopic to an embedding if and only if the Hatcher-Quinn invariant
\begin{equation}
\mu(f)\in \Omega_1\big( P(f,f)_{\Z_2} ; \zeta_{\Z_2}\big)
\end{equation}
vanishes. We propose to investigate the vanishing of $\mu(f)$ using the Gysin sequence of the double cover $P(f,f)\simeq E\Z_2\times P(f,f)\to P(f,f)_{\Z_2}$ (see Section 2). We abbreviate $P= P(f,f)$, and look at the portion of this sequence
\begin{equation}\label{gysinseq}
\xymatrix{
\cdots \ar[r] &  \Omega_1\big(P;\zeta\big) \ar[r]^-{i_*} &  \Omega_1 \big(P_{\Z_2} ; \zeta_{\Z_2}\big) \ar[r]^-{e}  & \Omega_0\big(P_{\Z_2}; \zeta_{\Z_2}-\lambda\big) \ar[r] & \cdots}
\end{equation}
Here $\lambda$ is the line bundle associated to the double cover $E\Z_2\times P\to P_{\Z_2}$, and $e$ is the Euler mapping.

\begin{prop}\label{euler} Let $f\co N^n\imm M^{2n-1}$ be a self-transverse immersion, where $N$ is closed and simply-connected, $M$ is connected and $n\geq 3$. Consider the relaxed immersion
\begin{equation}
g= (f,0)\co N\imm M\times \R,\qquad g(n) = \big( f(n), 0\big).
\end{equation}
Then $e\big(\mu(f)\big) = 0$ if and only if $\mu(g) = 0$ if and only if $g$ is regularly homotopic to an embedding.
\end{prop}

\begin{proof}
Let $h\co N\to \R$ be a smooth function such that the immersion
\begin{equation}
g'\co N\imm M\times \R,\qquad g'(n)=\big( f(n), h(n)\big)
\end{equation}
is self-transverse. Note that $g'$ is regularly homotopic to $g$, and so $\mu(g')=\mu(g)$. The anti-symmetric mapping
\begin{equation}\label{antisym}
\overline\varphi\co \overline{\Sigma}(f)\to \R,\qquad \overline\varphi (x,y)= h(x) -h(y)
\end{equation}
defines a section $\varphi\co \Sigma(f)\to \overline\Sigma(f)\times_{\Z_2} \R$ of the line bundle $\lambda_\pi$ associated to the double cover $\pi\co \overline\Sigma(f)\to \Sigma(f)$. The self-transversality of $g'$ implies that $\varphi$ is transverse to the zero section, and the zeroes of $\varphi$ are exactly the double points $\Sigma(g')\subseteq \Sigma(f)$ of $g'$.

Recall that $\mu(f) = [\Sigma(f), \phi,\Phi]$. The line bundle $\lambda_\pi$ can be identified with the pullback $\phi^*\lambda$. It follows from the description of the Euler mapping in Section 2 that
\begin{equation}
e\big( \mu(f)\big) = [\Sigma(g'), \phi|_{\Sigma(g')}, \Phi|_{\Sigma(g')}]\in \Omega_0(P_{\Z_2} ; \zeta_{\Z_2}- \lambda).
\end{equation}

We now investigate $\mu(g')$. The homotopy pullback $P'=P(g',g')$ comes with maps $p_1',p_2'\co P'\to N$, $p'\co P'\to M\times \R$, and
\begin{equation}
\mu(g')\in \Omega_0(P'_{\Z_2} ; \zeta'_{\Z_2}),\quad\mbox{where}\quad \zeta' = p_1'^*TN\oplus p_2'^*TN - p'^*T(M\times\R).
 \end{equation}
There is a canonical $\Z_2$-equivariant homotopy equivalence $\Pi\co P'\to P$ which projects a path in $M\times \R$ onto a path $M$, such that
\begin{equation}
p_1\Pi = p_1',\quad p_2\Pi=p_2',\quad\mbox{and}\quad p\Pi = \mathrm{pr}\, p',\quad\mbox{where }\mathrm{pr}\co M\times \R\to M.
\end{equation}
Let $\varepsilon^-_P$ be the trivial line bundle over $P$ with $\Z_2$-action $((x,\gamma,y),v)\mapsto ((y,\overline\gamma, x), -v)$, and note that $\lambda = (\varepsilon^-_P)_{\Z_2}$. It follows that $\Pi^*(\zeta_{\Z_2} - \lambda)\cong \zeta'_{\Z_2}$, and that $\Pi$ induces an isomorphism
\begin{equation}
\Pi_*\co \Omega_0\big( P'_{\Z_2}; \zeta'_{\Z_2}\big) \stab \Omega_0\big( P_{\Z_2}; \zeta_{\Z_2}-\lambda \big).
\end{equation}
We claim that $e\big( \mu(f)\big) = \Pi_*\mu(g')$, and hence that $e\big( \mu(f)\big)$ vanishes if and only if $\mu(g')=\mu(g)$ vanishes. By definition,
\begin{equation}
\mu(g') = [\Sigma(g'),\phi', \Phi']\in \Omega_0(P_{\Z_2} ; \zeta_{\Z_2}-\lambda),
\end{equation}
where $\phi'(x,y) = [e(x,y), (x, c_{g'(x)},y)]$ and
\begin{equation}
\Phi'\co T\Sigma(g')\oplus \phi'^*(p'^*T(M\times\R))_{\Z_2} \stab \phi'^*(p_1'^*TN\oplus p_2'^*TN)_{\Z_2}.
\end{equation}
The claim can be proved by noting that $\Pi\phi' = \phi|_{\Sigma(g')}$, and making the identifications
\begin{equation}
 p'^*T(M\times\R) \cong p'^*\mathrm{pr}^*(TM\oplus\varepsilon^-_M) \cong  \Pi^*(p^*TM\oplus\varepsilon^-_P),
\end{equation}
\begin{equation}
 p_1'^*TN\oplus p_2'^*TN = \Pi^*(p_1^* TN\oplus p_2^*TN),
 \end{equation}
 and noting that both $\Phi|_{\Sigma(g')}$ and $\Phi'$ arise from consideration of the embedding $i'\co \overline\Sigma(g')\hookrightarrow N\times N$.
\end{proof}

The question of when the relaxed immersion $g\co N\imm M\times\R$ is regularly homotopic to an embedding has been considered by Wall \cite[p.83]{Wa99} and by Sz\H ucs \cite[p.252]{Sz}, and turns out to depend on the nature of the double circles of the original immersion $f$.

Recall that the construction of Wall's invariant requires choosing an ordering of each double point. In the case of $f\co N^n\imm M^{2n-1}$, the double points are replaced by finitely many double circles $C\subseteq M$. Each double circle $C$ is doubly covered by its pre-image $\overline{C}\subseteq N$. Let us call a double circle $C$ {\em trivial} if the corresponding cover $\pi_C\co \overline{C}\to C$ is trivial, and {\em non-trivial} otherwise. An ordering of the double points now corresponds to a section of $\pi_C$ over each trivial double circle. As the next proposition shows, non-trivial double circles give a first obstruction to $f$ being regularly homotopic to an embedding.

\begin{prop}\label{Szucs}
Let $f\co N^n\imm M^{2n-1}$ be as in the statement of Proposition \ref{euler}. Then the relaxed immersion
\begin{equation}
g= (f,0)\co N\imm M\times \R,\qquad g(n) = \big( f(n), 0\big)
\end{equation}
is regularly homotopic to an embedding if and only if the number of non-trivial double circles of $f$ in each one-dimensional homotopy class of $M$ is even.
\end{prop}

\begin{proof}
As in the proof of Proposition \ref{euler}, choose a smooth function $h\co N\to \R$ such that $g'=(f,h)\co N\imm M\times \R$
is self-transverse, and note that any such $g'$ is regularly homotopic to $g$. Hence $g$ is regularly homotopic to an embedding if and only if $0=\mu_{W}(g')\in H_0(\Z_2 ; \Z[\pi_1(M\times\R)])=H_0(\Z_2 ; \Z[\pi])$, by Theorem \ref{Wall}. The rest of the proof consists of an analysis of when the non-simply-connected Whitney trick can be applied to the relaxed immersion $g'$.

The double points of $g'$ coincide with the zeroes of the section of $\lambda_\pi$ induced by the anti-symmetric mapping
$\overline\varphi\co \overline{\Sigma}(f)\to \R$ (see (\ref{antisym})). We may assume $h$ was chosen to separate the two components of $\overline{C}$ for each trivial double circle $C\subseteq\Sigma(f)$. It follows that all the double points of $g'$ lie on non-trivial double circles, and that {\em the number of double points on each circle is odd.}

Let $[x,y]$ be a double point of $g'$ lying on a non-trivial double circle $C$ of $f$. Then $\sigma_{(x,y)}\in\pi_1(M,m_0)$ is the homotopy class of a path which travels from $m_0$ to $C$ along the image under $f$ of a path in $N$, then around $C$, then back to $m_0$ along the same path.
Clearly $\sigma_{(y,x)} = {\sigma_{(x,y)}}^{-1} = \sigma_{(x,y)}$ (since $C = f(\overline{C})$ and $N$ is simply-connected).

We now apply Wall's formula
 \begin{equation}
 \lambda(g',g') = \tilde\mu_{W}(g') + \overline{\tilde\mu_{W}(g')} + \chi(g')\in\Z[\pi],
 \end{equation}
 where $\lambda(g',g')\in\Z[\pi]$ is the (non-simply-connected) intersection number of $g'$ with a transverse approximation of $g'$, and $\chi(g')$ denotes the Euler number of $\nu_{g'}$ (see \cite[Theorem 5.2]{Wa99}). However $\lambda(g',g')=0$ (since $N$ is compact, the $\R$ coordinate in $M\times \R$ allows us to separate the two copies of $g'$) and $\chi(g')=0$ (since $\nu_{g'}\cong\nu_f\oplus\varepsilon^1$). Therefore
\begin{equation}
0 = \sum_{[x,y]\in \Sigma(g')} (\varepsilon_{(x,y)}+\varepsilon_{(y,x)})\sigma_{(x,y)}.
\end{equation}
Let $[x,y]\in \Sigma(g')$ with $\varepsilon_{(x,y)}=\varepsilon_{(y,x)}$. Then there must be another double point $[x',y']\in\Sigma(g')$ with $\varepsilon_{(x',y')}=\varepsilon_{(y',x')}=-\varepsilon_{(x,y)}$ and $\sigma_{(x',y')} = \sigma_{(x,y)}$ to cancel it. In this case these two double points contribute $0$ to $\tilde\mu_W(g')$, so may be ignored. The remaining double points have $\varepsilon_{(x,y)}= -\varepsilon_{(y,x)} = -(-1)^n w(\sigma_{(x,y)})\varepsilon_{(x,y)}$, and consequently $2[\sigma_{(x,y)}]=0\in H_0(\Z_2 ; \Z[\pi])$. Hence
\begin{equation}
\mu_W(g') = \left[ \sum_{[x,y]\in\Sigma(g')} \sigma_{(x,y)}\right]
\end{equation}
is zero if and only if the number of double points in each homotopy class is even. Since the number of double points on each non-trivial double circle is odd, the proposition follows.
\end{proof}

 Now if $n\geq 3$ and $f\co N^n \imm M^{2n-1}$ is an immersion such that $g\co N \imm M\times \R$ is regularly homotopic to an embedding, then the sequence (\ref{gysinseq}) tells us that the Hatcher-Quinn invariant $\mu(f)$ lifts non-uniquely to an element
\begin{equation}
 \overline{\mu}(f)\in \Omega_1( P(f,f) ; \zeta).
\end{equation}
The vanishing of $\overline{\mu}(f)$ is a sufficient condition for $f$ to be regularly homotopic to an immersion. The next result identifies the group $\Omega_1( P(f,f) ; \zeta)$ when $N$ is $2$-connected.

\begin{prop}
Let $f\co N^n\imm M^{2n-1}$, where $N$ is $2$-connected and $M$ is connected. There is an isomorphism of abelian groups
\begin{equation}\label{groups}
\chi\co\Omega_1(P(f,f) ;\zeta) \stab \bigoplus_{\sigma\in\pi_1(M,m_0)}\Z_2\times\pi_2(M,m_0).
\end{equation}
\end{prop}
\begin{proof}
Since $N$ is $2$-connected, the fibre inclusion $\iota\co \Lambda M \to P(f,f)$ induces an isomorphism $\iota_*\co \pi_i(\Lambda M)\to \pi_i(P(f,f))$ for $i=0,1$, and hence an isomorphism
\begin{equation}
\iota_*\co \Omega_1(\Lambda  M; \iota^*\zeta) \stab \Omega_1(P(f,f); \zeta).
\end{equation}
As noted in the proof of Lemma \ref{modules}, the virtual bundle $\iota^*\zeta$ is trivial. Hence we have isomorphisms
\begin{equation}
\Omega_1(\Lambda M;\iota^*\zeta)\cong \Omega_1^{fr}(\Lambda M) \cong \bigoplus_{\sigma\in\pi} \Omega_1^{fr}(\Lambda_\sigma M),
\end{equation}
where $\pi=\pi_1(M,m_0)$ and $\Omega_*^{fr}$ denotes the unreduced homology theory given by framed bordism (see \cite{Ko} or \cite{Dax}). The second isomorphism is given by the disjoint union axiom. The Atiyah-Hirzebruch spectral sequence for framed bordism gives a short exact sequence
\begin{equation}
0\to \Z_2 = \Omega^{fr}_1(*)\to  \Omega^{fr}_1(\Lambda_\sigma M) \to H_1(\Lambda_\sigma M ; \Z) \to 0
\end{equation}
which is split by the constant map $\Lambda_\sigma M\to *$. Since each path component $\Lambda_\sigma M$ is homotopy equivalent to the component $\Lambda_0 M$ of the constant loop, we therefore have isomorphisms
\begin{equation}
\Omega_1(P(f,f); \zeta)\cong\bigoplus_{\sigma\in\pi}\Z_2\times H_1(\Lambda_0 M ; \Z) \cong \bigoplus_{\sigma\in\pi} \Z_2\times \pi_2(M,m_0)
\end{equation}
(by the Hurewicz homomorphism and the fact that $\pi_1(\Lambda_0 M,m_0)\cong\pi_2(M,m_0)$ is abelian).
\end{proof}
 Hence when $n\geq 5$ and $N$ is $2$-connected, the image of a lift $\overline{\mu}(f)$ in a certain quotient of $\bigoplus_{\sigma\in\pi}\Z_2\times \pi_2(M,m_0)$ defines a complete obstruction to $f\co N^n\imm M^{2n-1}$ being regularly homotopic to an embedding.

\end{document}